\documentclass[12pt]{amsart}

\usepackage{graphicx}
\usepackage{amsmath}
\usepackage{mathtools}
\usepackage{amssymb}
\usepackage{amsfonts}
\usepackage{amsthm}
\usepackage{cancel}
\usepackage[all]{xy}
\usepackage[neveradjust]{paralist}
\usepackage{enumitem}
\usepackage{multicol}
\usepackage{mathrsfs}
\usepackage{mathdots}
\usepackage[pagebackref,colorlinks=true,linkcolor=blue,citecolor=blue]{hyperref}
\usepackage{tikz-cd}
\usepackage{tikz}
\usepackage{blkarray}
\usepackage{scalerel,stackengine}

\UseRawInputEncoding

\stackMath
\newcommand\reallywidehat[1]{%
\savestack{\tmpbox}{\stretchto{%
  \scaleto{%
    \scalerel*[\widthof{\ensuremath{#1}}]{\kern-.6pt\bigwedge\kern-.6pt}%
    {\rule[-\textheight/2]{1ex}{\textheight}}
  }{\textheight}%
}{0.5ex}}%
\stackon[1pt]{#1}{\tmpbox}%
}

\tikzset{
  symbol/.style={
    draw=none,
    every to/.append style={
      edge node={node [sloped, allow upside down, auto=false]{$#1$}}}
  }
} 

\usepackage{color}
\usepackage{blkarray}

\newcommand{\Z}{\mathbb{Z}}

\newcommand{\R}{\mathbb{R}}
\newcommand{\BC}{\mathbb{C}}

\newcommand{\SL}{\mathrm{SL}}
\newcommand{\GL}{\mathrm{GL}}
\newcommand{\SO}{\mathrm{SO}}
\newcommand{\Sp}{\mathrm{Sp}}

\newcommand{\RG}{\mathrm{G}}

\newcommand{\Mseg}{\underline{\mathrm{Mseg}}}

\newcommand{\Seg}{\underline{\mathrm{Seg}}}

\newcommand{\half}[1]{\frac{#1}{2}}

\newcommand{\comment}[1]{}

\newtheorem{thm}{Theorem}[section]
\newtheorem{cor}[thm]{Corollary}
\newtheorem{lemma}[thm]{Lemma}
\newtheorem{prop}[thm]{Proposition}
\newtheorem {conj}[thm]{Conjecture}

\newtheorem {ques/conj}[thm]{Question/Conjecture}

\newtheorem{defn}[thm]{Definition}

\newtheorem{exmp}[thm]{Example}

\newtheorem{algo}[thm]{Algorithm}

\newtheorem*{globalcond*}{Global Condition}
\newtheorem*{localcond*}{Local Condition}

\newtheorem*{globalconj*}{Global Conjecture}
\newtheorem*{localconj*}{Local Conjecture}

\newtheorem*{nonzero*}{Conjecture on the non-vanishing of the normalized intertwining operators}
\newtheorem*{holo*}{Conjecture on the holomorphicity of the normalized intertwining operators}

\DeclareMathOperator{\supp}{supp}

\DeclareMathOperator{\Sym}{Sym}

\numberwithin{equation}{section}

\let\oldbullet\bullet
\renewcommand{\bullet}{{\vcenter{\hbox{\tiny$\oldbullet$}}}}

\begin{document}
\title[A Combinatorial Lemma]{Vogan's Conjecture on local Arthur packets of $p$-adic $\GL_n$ and a combinatorial Lemma}

\author[C.-H. Lo]{Chi-Heng Lo}
\address{Department of Mathematics\\
Purdue University\\
West Lafayette, IN, 47907, USA}
\email{lo93@purdue.edu}

\subjclass[2000]{Primary 11F70, 22E50; Secondary 11F85}


\keywords{Local Arthur Packets, Local Arthur parmeters, Vogan's conjecture}

\thanks{The research of the author is partially supported by the NSF Grant DMS-1848058.}

\begin{abstract}
    For $\GL_n$ over a $p$-adic  field, Cunningham and Ray proved Vogan's conjecture, that is, local Arthur packets are the same as ABV packets. They used the endoscopic theory to reduce the general case to a combinatorial lemma for irreducible local Arthur parameters, and their proof implies that one can also prove Vogan's conjecture for $p$-adic $\GL_n$ by proving a generalized version of this combinatorial lemma. Riddlesden recently proved this generalized lemma. In this paper, we give a new proof of it, which has its own interest.  
\end{abstract}

\maketitle

\section{Introduction}

Let $F$ be a non-Archimedean field of characteristic 0 and denote $W_F$ the Weil group of $F$. Let $\RG$ be a connected reductive group defined over $F$. We denote $G:=\RG(F)$ and $\Pi(G)$ the isomorphism classes of smooth irreducible representations of $G$. A local Arthur parameter $\psi$ is a continuous homomorphism
\[ \psi : W_F \times \SL_2^D(\BC) \times \SL_2^A(\BC) \to {}^{L}G,\]
such that
\begin{enumerate}
    \item [(1)] the restriction of $\psi$ to $W_F$ has bounded image;
    \item [(2)] the restrictions of $\psi$ to both $\SL_2(\BC)$ are analytic;
    \item [(3)] $\psi$ commutes with the projections $W_F \times \SL_2^D(\BC) \times \SL_2^A(\BC) \to W_F$ and ${}^{L}G \to W_F$.
\end{enumerate}
Here the first $\SL_{2}^{D}(\BC)$ is called the Deligne-$\SL_2$ and the second $\SL_{2}^{A}(\BC)$ is called the Arthur-$\SL_2$.

In the fundamental work \cite{Art13}, for a local Arthur parameter $\psi$ of quasi-split classical groups, Arthur attached a local Arthur packet $\Pi_\psi$. This is a finite multi-set of smooth irreducible representations, satisfying certain twisted endoscopic character identities (\cite[\S2]{Art13}). 
Assuming the Ramanujan Conjecture, Arthur showed that these local Arthur packets characterize the local components of discrete square-integrable automorphic representations.

An interesting question is to construct each local Arthur packet $\Pi_{\psi}$ over $p$-adic fields besides the abstract definition in \cite{Art13}. In a series of work (\cite{Moe06a, Moe06b, Moe09a, Moe10, Moe11a}), M{\oe}glin explicitly constructed each local Arthur packet $\Pi_{\psi}$ and showed that it is multiplicity free. However, there are difficulties in her construction when trying to compute the representations in the local Arthur packets using the Langlands classification. To remedy this, for symplectic or split odd special orthogonal groups, Atobe gave a reformulation of M{\oe}glin's construction (\cite{Ato20b}) based on the derivatives introduced in \cite{AM20} and gave an algorithm to explicitly compute the Langlands classification for the representations in a local Arthur packet. The main tools in these results are (partial) Aubert-Zelevinsky involution and partial Jacquet module, which are representation theoretic.

On the other hand, in \cite{CFMMX22}, Cunningham \textit{et al.} aim to construct local Arthur packets over $p$-adic fields using geometric approach. They extend the work of \cite{ABV92} to $p$-adic reductive groups and defined a packet $\Pi_\phi^{\mathrm{ABV}}$ using micro-local vanishing cycle functors, for any $L$-parameter $\phi$ of any $p$-adic reductive group $G$. 
Each ABV-packet $\Pi_{\phi}^{\textrm{ABV}}$ consists of not only the representations of $G$, but also the representations of the inner forms of $G$. We shall denote $\Pi_{\phi}^{\textrm{ABV}}(G):= \Pi_{\phi}^{\textrm{ABV}} \cap \Pi(G)$. 
It is expected that the ABV-packets recover the local Arthur packets in the following sense. For each local Arthur parameter $\psi$ of $G$, we associate an $L$-parameter $\phi_{\psi}$ of $G$ by
\[ \phi_{\psi}(w,x):= \psi\left(w,x, \begin{pmatrix}
    |w|^{\half{1}} & \\ & |w|^{\half{-1}}
\end{pmatrix} \right).\]
The \emph{Vogan's Conjecture} is stated as follows. 

\begin{conj}[{\cite[Conjecture 8.3.1(a)]{CFMMX22}}] \label{conj Vogan}
Let $\psi$ be a local Arthur parameter of $\RG(F)$ and denote $\phi_{\psi}$ the associated $L$-parameter. The following equality holds.
\[ \Pi_{\psi}= \Pi_{\phi_{\psi}}^{\textrm{ABV}} (\RG(F))\]
\end{conj}
There are more precise statements matching the distributions on the Arthur side and the ABV side. We refer to \cite[Conjecture 8.3.1]{CFMMX22} for more details.

 Conjecture \ref{conj Vogan} remains widely open. The only known case is $\GL_n(F)$, proved by Cunningham and Ray in \cite{CR22, CR23}.
 We roughly describe their main idea as follows. A crucial ingredient of their proof is the combinatorial description of an involution on $L$-parameters. We shall call this map the \emph{Pyasetskii involution} and denote it by
\[ \phi \mapsto \widehat{\phi}.\]
This involution is defined via the geometric structure on the Vogan variety. We refer to \cite[\S 6.4]{CFMMX22} or \cite[\S 4.3]{Zel81} for the precise definition.

When $G=\GL_n(F)$, there are bijections among irreducible representations of $G$, $L$-parameters of $G$, and a collection of multi-segments (see \S \ref{sec multi-segments} for details). In \cite{MW86},  M{\oe}glin and Waldspurger showed that under these bijections, Pyasetskii involution on $L$-parameters matches the \emph{Zelevinsky involution} on irreducible representations defined in \cite[\S 4.1]{Zel81}. Moreover, they gave a combinatorial algorithm on multi-segments to realize these involutions, for more details see \S \ref{sec M-W}. Later in \cite{KZ96}, Knight and Zelevinsky gave a closed formula for the involution on multi-segments, which is proved by using the theory of flows in network .



For $\GL_n(F)$, the structure of $L$-packets and local Arthur packets are simple: they are all singletons. Therefore, we have $\Pi_{\psi} \subseteq \Pi_{\phi_{\psi}}^{\textrm{ABV}}$ for free. With this observation and the geometric structure of ABV-packets of $\GL_n(F)$ (\cite[Proposition 3.2.1]{CFK22}), Cunningham and Ray demonstrated in the proof of \cite[Theorem 5.3]{CR22} that for an irreducible local Arthur parameter $\psi$ of $\GL_n(F)$, the equality 
\[\Pi_{\psi}= \Pi_{\phi_{\psi}}^{\textrm{ABV}}(\GL_n(F))\]
holds if the following lemma holds for $\psi$.
\begin{lemma}[{\cite[Lemma 4.8]{CR22}}]\label{lem main irred intro}
    Let $\psi$ be an irreducible local Arthur parameter of $\GL_n(F)$ and denote $\phi_{\psi}$ the associated $L$-parameter. If $\phi$ is an $L$-parameter of $\GL_n(F)$ satisfying that $\phi \geq \phi_{\psi}$ and $\widehat{\phi} \geq \widehat{\phi_{\psi}}$, then $\phi=\phi_{\psi}$.    
\end{lemma}
Here the inequality is under the closure ordering in the associated Vogan variety, which is equivalent to the partial ordering on multi-segments considered in \cite{Zel81} (see \S \ref{sec partial order} for details). They proved above lemma, and hence established the Vogan's Conjecture in this case. Later in \cite{CR23}, they used endoscopic lifting to reduce the general case to the case of irreducible parameters. This proved Vogan's Conjecture for $\GL_n(F)$ completely.

On the other hand, the proof in \cite[Theorem 5.3]{CR22} implies that for an \emph{arbitrary} local Arthur parameter $\psi$ of $\GL_n(F)$, not necessarily irreducible, the equality 
\[\Pi_{\psi}= \Pi_{\phi_{\psi}}^{\textrm{ABV}}(\GL_n(F))\]
holds if the following generalized lemma holds for $\psi$.
\begin{lemma}\label{lem main intro}
    Let $\psi$ be an arbitrary local Arthur parameter of $\GL_n(F)$ and denote $\phi_{\psi}$ the associated $L$-parameter. If $\phi$ is an $L$-parameter of $\GL_n(F)$ satisfying that $\phi \geq \phi_{\psi}$ and $\widehat{\phi} \geq \widehat{\phi_{\psi}}$, then $\phi=\phi_{\psi}$.    
\end{lemma}
In \cite{Rid23}, Riddlesden proved Lemma \ref{lem main intro}, mainly using the network description of the Zelevinsky involution as in \cite{KZ96}. Therefore, combining with the proof of \cite[Theorem 5.3]{CR22}, this provides a combinatorial approach to Vogan's Conjecture for $\GL_n(F)$.

In this paper, we give a new proof of Lemma \ref{lem main intro}, and hence provide another approach to Vogan's Conjecture for $\GL_n(F)$. In comparison with \cite{Rid23}, our proof is elementary and only involves the M{\oe}glin-Waldspurger algorithm, hence has its own interest. More precisely, our proof of Lemma \ref{lem main intro} is reduced to Proposition \ref{prop main} which is motivated by the proof of \cite[Theorem 1.16]{HLLZ22}. Roughly speaking, let $\pi$ be an irreducible representation of $\Sp_{2n}(F)$ or split $\SO_{2n+1}(F)$ of Arthur type. If $\psi$ is the ``most tempered" local Arthur parameter in the set 
\[ \Psi(\pi):=\{ \psi \ | \ \pi \in \Pi_{\psi}\},\]
see \cite[\S 10]{HLL22} for precise definition, then $\phi_{\pi}$, the $L$-parameter of $\pi$, must share a common summand(s) with $\phi_{\psi}$ that corresponds to some term $\rho\otimes \Sym^{d_i} \otimes \Sym^{a_i}$ in the decomposition 
\[    \psi= \bigoplus_{\rho} \bigoplus_{i\in I_{\rho}} \rho \otimes \Sym^{d_i} \otimes \Sym^{a_i},\]
 where the pair ($d_i,a_i$) is ``extremal" in certain sense. Moreover, if we define $\pi^{-}$, whose $L$-parameter is obtained from $\phi_{\pi}$ by removing this common summand(s), then $\pi^{-}$ is still of Arthur type, and $\pi^{-} \in \Pi_{\psi^{-}}$ where $\psi^{-}$ is obtained from $\psi$ by a similar process.
The above intuition degenerates a lot in the case of $\GL_n(F)$. Namely, the set $\Psi(\pi)$ is always a singleton so we don't need to worry which $\psi$ is the ``most tempered" one. Also, not only sharing a common summand, $\phi_\psi$ and $\phi_{\pi}$ are always identical. The question then becomes to figure out what is the ``extremal" condition for the pair ($d_i,a_i$) that can be detected by the inductive process in the M{\oe}glin-Waldspurger Algorithm. This gives the definition of $a$ and $d$ in Proposition \ref{prop main}.

Following is the structure of this paper. In \S \ref{sec preliminary}, we recall the necessary notation and preliminaries. We recall the notion of multi-segments in \S\ref{sec multi-segments}, the partial ordering on multi-segments in \S\ref{sec partial order}, and the M{\oe}glin-Waldspurger Algorithm in \S\ref{sec M-W}. In \S \ref{sec lemma MW}, we rephrase M{\oe}glin-Waldspurger Algorithm and develop certain notations and lemmas for the proofs in \S \ref{sec proof}. Then we prove Lemma \ref{lem main intro} in \S \ref{sec proof}. We rephrase Lemma \ref{lem main intro} in terms of multi-segment in \S \ref{sec rephrase}. Prove the key of reduction, Proposition \ref{prop main}, in \S \ref{sec reduction}. Finally we prove Lemma \ref{lem main intro} in \S\ref{sec finish}.

\subsection*{Acknowledgements} 

I would like to thank  Professor Baiying Liu and Professor Freydoon Shahidi for the constant support and encouragement. I also would like to thank Professor Clifton Cunningham, Mishty Ray, and Connor Riddlesden for helpful communications. 

\section{Preliminaries}\label{sec preliminary}
Denote $F$ a non-Archimeden field of characteristic 0 and $W_F$ the Weil group of $F$. Let $|\cdot|$ be the normalized absolute value of $F$, and also regarded as a character of $\GL_n(F)$ by composing with determinant.

We denote $\Pi(\GL_n(F))$ the isomorphism classes of irreducible smooth representations of $\GL_n(F)$, and denote $\Phi(\GL_n(F))$ the equivalence class of $L$-parameters of $\GL_n(F)$. Also, we let 
\[ \Pi(\GL(F)):= \bigsqcup_{n \geq 1} \Pi(\GL_n(F)),\  \Phi(\GL(F)):= \bigsqcup_{n \geq 1} \Phi(\GL_n(F)) \]

We denote $+$ the sum of multi-set (disjoint union), and $\setminus$ the difference of multi-set.
\subsection{\texorpdfstring{Langlands classification for $\GL_n(F)$}{}}\label{sec multi-segments}
In this subsection, we recall the Langlands classification for $\GL_n(F)$ and the bijection among $\Pi(\GL_n(F))$, $\Phi(\GL_n(F))$ and multi-segments of correct rank.

Let $\mathcal{C}(\GL_n(F))$ denote the isomorphism classes of supercuspidal representations of $\GL_n(F)$. By Local Langlands Correspondence of $\GL_n(F)$, we may also identify $\mathcal{C}(\GL_n(F))$ as the isomorphism classes of $n$-dimensional irreducible representations of $W_F$. We let 
\[ \mathcal{C}:= \bigsqcup_{n \geq 1} \mathcal{C}(\GL_n(F)),\]
and denote $\mathcal{C}_{\textrm{unit}}$ be the subset of $\mathcal{C}$ consists of unitary supercuspidal representations.

Let $P$ be a standard parabolic subgroup of $\GL_n(F)$ with Levi subgroup $L \cong \GL_{n_1}(F) \times \cdots \times \GL_{n_s}(F)$. An irreducible representation $\sigma$ of $L$ can be identified with 
\[\sigma = \sigma_1 \otimes \cdots \otimes \sigma_s,\]
where $\sigma_i \in \Pi(\GL_{n_i}(F))$. We denote the normalized parabolic induction $\textrm{Ind}_{P}^{\GL_n(F)} \sigma$ by
\[ \sigma_1 \times \cdots \times \sigma_s.\]

A segment $\Delta$ is a set
\[ \{ \rho|\cdot|^{b}, \rho|\cdot|^{b+1},\dots, \rho|\cdot|^{e}\},\]
where $\rho \in \mathcal{C}_{\textrm{unit}}$, $b,e \in \R$ such that $e-b \in \Z_{\geq 0}$. We shall denote $\Delta=[b,e]_{\rho}$ and call $b$ the base value of $\Delta$ and $e$ the end value of $\Delta$, and $e-b+1$ the length of $\Delta$. We also write
\[ b(\Delta):= b,\ e(\Delta):= e,\ l(\Delta):=e-b+1.\]
A multi-segment, which we usually denote by $\alpha, \beta ,\gamma$ or $\delta$, is a finite multi-set of segments. We denote the collection of segments by $\Seg$ and the collection of multi-segments $\Mseg$. For each $\rho \in \mathcal{C}_{\text{unit}}$, let $\Seg_{\rho}$ denote the subset of $\Seg$ consists of segments of the form $[b,e]_{\rho}$, and let $\Mseg_{\rho}$ denote the subset of $\Mseg$ consists of multi-sets of segments in $\Seg_{\rho}$. For each $\alpha \in \Mseg$, there is a unique decomposition
\[ \alpha= \sum_{\rho \in \mathcal{C}_{\textrm{unit}}} \alpha_{\rho},\]
with $\alpha_{\rho} \in \Mseg_{\rho}$ and $\alpha_{\rho}= \emptyset$ except a finite number of $\rho \in \mathcal{C}_{\textrm{unit}}$.

For each segment $[y,x]_{\rho}$, let $\Delta_{\rho}[x,y]$ be the unique irreducible subrepresentation of the parabolic induction
\[ \rho|\cdot|^{x} \times \rho|\cdot|^{x-1} \times \cdots \times \rho|\cdot|^{y}. \]

The Langlands classification of $\GL_n(F)$ states the following. Any representation $\pi\in \Pi(\GL_n(F))$ can be realized as the unique irreducible subrepresentation of a parabolic induction
\[ \Delta_{\rho_1}[x_1,y_1]\times \cdots \times \Delta_{\rho_f}[x_f, y_f], \]
where
\begin{enumerate}
    \item [$\oldbullet$] $n= \sum_{i=1}^f \dim(\rho_i) (x_i-y_i+1)$,
        \item [$\oldbullet$] $\rho_i \in \mathcal{C}_{\textrm{unit}}$ , and
    \item [$\oldbullet$] $x_1+y_1 \leq \dots \leq x_f+y_f$.
\end{enumerate}
Here $\dim(\rho_i)$ is the dimension of $\rho_i$ as an irreducible representation of $W_F$. Moreover, the multi-set 
\[\{ \Delta_{\rho_1}[x_1,y_1],\ldots, \Delta_{\rho_f}[x_f, y_f] \}\]
with above requirement is unique.  With these notations, we may write down the $L$-parameter of $\pi$ as 
\[ \phi_\pi= \rho_1|\cdot|^{\half{x_1+y_1}} \otimes \Sym^{x_1-y_1}  \oplus \cdots \oplus \rho_f|\cdot|^{\half{x_f+y_f}} \otimes \Sym^{x_f-y_f},\]
where $\Sym^{b+1}$ is the unique $b$-dimensional irreducible analytic representation of $\SL_2(\BC)$. Also, we associate the following multi-segment to $\pi$.
\[ \delta_\pi:= \{ [y_1,x_1]_{\rho_1},\dots, [y_f,x_f]_{\rho_f}\}\]
Thus with above notation, we have the following correspondence.
\[
\begin{tikzcd}
    \Pi(\GL(F))\ar[r,leftrightarrow]& \Phi(\GL(F))\ar[r,leftrightarrow] & \Mseg\\
    \pi \ar[r,mapsto]&  \phi_{\pi} \ar[r,mapsto]& \delta_\pi.
\end{tikzcd}
\]
Suppose $\phi$ is an $L$-parameter of $\GL_n(F)$. We also denote $\delta_{\phi}:=\delta_{\pi}$ where $\pi$ is the unique representation of $\Pi(\GL_n(F))$ such that $\phi_{\pi}=\phi$.

\subsection{A partial ordering on multi-segments}\label{sec partial order}
In this subsection, we recall the partial ordering on $\Mseg$ introduced in \cite{Zel81}.

Suppose $\Delta_1=[b_1,e_1]_{\rho_1}$ and $\Delta_2=[b_2,e_2]_{\rho_2}$ are two segments. We say $\Delta_1$ and $\Delta_2$ are \emph{linked} if the union $\Delta_1 \cup \Delta_2$ (as a set) is also a segment, and $\Delta_1 \not\supseteq \Delta_2 $, $\Delta_2 \not\supseteq \Delta_1 $. In particular, $\Delta_1$ and $\Delta_2$ are linked only if $\rho_1 \cong \rho_2$ (recall that we require $\rho_1,\rho_2$ to be unitary). 

Now let $\alpha, \beta$ be two multi-segments. We say $\beta$ is obtained from $\alpha$ by performing a single \emph{elementary operation} if we can form $\beta$ from $\alpha$ by replacing a sub-multi-set $\{\Delta_1, \Delta_2\}$ of $\alpha$ by 
 \[\begin{cases} 
   \{ \Delta_1, \Delta_2\} &\text{ if } \Delta_1, \ \Delta_2 \text{ are not linked},\\
   \{ \Delta_1 \cup \Delta_2,\  \Delta_{1} \cap \Delta_2\} &\text{ if }\Delta_1,\ \Delta_2 \text{ are linked and } \Delta_1 \cap \Delta_2 \neq \emptyset,\\
    \{ \Delta_1 \cup \Delta_2 \} &\text{ if }\Delta_1,\ \Delta_2 \text{ are linked and } \Delta_1 \cap \Delta_2 = \emptyset.
    \end{cases}\]
\begin{defn}\label{def partial order}
    Let $\alpha, \beta$ be two multi-segments. We define $\alpha \geq \beta$ if $\beta$ can be obtained from $\alpha$ by performing a sequence of elementary operations. This gives a partial ordering on $\Mseg$.
\end{defn}

Suppose $\delta= \{\Delta_1,\dots, \Delta_r\}$ is a multi-segment. We define
\[ \supp (\delta):= \sum_{i=1}^r \Delta_i,\]
which is a multi-set of supercuspidal representations. Then it is not hard to see that $\alpha \geq \beta$ only if $\supp(\alpha)=\supp(\beta)$. Also, if $\alpha= \sum_{\rho \in \mathcal{C}_{\textrm{unit}}} \alpha_{\rho}$ and $\beta= \sum_{\rho \in \mathcal{C}_{\textrm{unit}}} \beta_{\rho}$, then $\alpha \geq \beta$ if and only if $\alpha_{\rho} \geq \beta_{\rho}$ for every $\rho.$

It is proved in \cite[\S 2]{Zel81} that the partial ordering in Definition \ref{def partial order} is exactly the closure ordering on the orbits of Vogan varieties.
\begin{thm}\cite[Theorem 2.2]{Zel81} \label{thm partial order}
    Let $\phi_1,\phi_2$ be two $L$-parameters of $\GL_n(F)$. The followings are equivalent.
    \begin{enumerate}
        \item [(a)]$\delta_{\phi_1} \geq \delta_{\phi_2}$.
        \item [(b)]$\phi_1 \geq \phi_2$.
    \end{enumerate}
Here the $\geq$ in Part (b) is the closure ordering on the associated Vogan variety.
\end{thm}

\subsection{M{\oe}glin-Waldspurger algorithm}\label{sec M-W}
In this subsection, we recall the statement of the M{\oe}glin-Waldspurger algorithm and introduce the related notations. For simplicity, we shall assume any multi-segment in this section is in $\Mseg_{\rho}$ for a fixed $\rho \in \mathcal{C}_{\textrm{unit}}$, and omit $\rho$ in the notation.

Suppose $\Delta_1=[b_1,e_1]$ and $\Delta_2=[b_2,e_2]$ are linked. We say $\Delta_1$ \emph{precedes} $\Delta_2$ if $b_1<b_2$ and $a_1<a_2$. For $\Delta=[b,e]$, we define \[\Delta^{-}:= \begin{cases}
    [b,e-1] & \textrm{ if }b\neq e,\\
    \emptyset & \textrm{ otherwise.}
\end{cases}\]

With these definitions, we are ready to state the M{\oe}glin-Waldspurger algorithm.

\begin{algo}[{M{\oe}glin}-Waldspurger Algorithm]\label{algo M-W}
Suppose $\alpha$ is a multi-segment. We associate a segment $M(\alpha)$ as follows.
\begin{enumerate}
    \item [(1)] Set $e$ to be the largest end value of segments in $\alpha$. Set $m:=e$.
    \item [(2)] Consider all segments in $\alpha$ with end value $m$. Among these, choose a segment with the
largest base value and call it $\Delta_{m}$.
\item [(3)] Consider the set of all segments in $\alpha$ that precede $\Delta_{m}$ with end value $m-1$. If this is empty,
go to step (5). Otherwise, choose a segment from this set with the largest base value and
call it $\Delta_{m-1}$.
\item [(4)]Set $m:=m-1$ and go to step(3).
\item [(5)] Return $M(\alpha)=[m,e]$.
\end{enumerate}
Next, we inherit the following notation from the above procedure. Define $\alpha\setminus M(\alpha)$ to be the multi-segment obtained from $\alpha$ by replacing $\Delta_{i}$ with $\Delta_i^{-}$ for all $m \leq  i \leq  e$ and removing empty sets.

Finally, we set 
\[\widetilde{\alpha}:=\{ M(\alpha)\}+( \widetilde{\alpha\setminus M(\alpha)}) \]
if $\alpha \neq \emptyset$ and $\widetilde{\alpha}:=\emptyset$ otherwise.
\end{algo}

For general multi-segment $\alpha= \sum_{\rho \in \mathcal{C}_{\textrm{unit}}} \alpha_{\rho}$, we define $\widetilde{\alpha}:= \sum_{\rho \in \mathcal{C}_{\textrm{unit}}} \widetilde{\alpha_{\rho}}$.

The main result of \cite{MW86} is that Algortihm \ref{algo M-W} computes the Zelevinsky invoultion on multi-segments and also the Pyasetskii involution on $L$-parameters of $\GL_n(F)$.
\begin{thm}\cite[Théorème II.13]{MW86}\label{thm M-W}
For any $L$-parameter $\phi$ of $\GL_n(F)$, we have 
\[ \delta_{\widehat{\phi}}= \widetilde{\delta_{\phi}}.\]
\end{thm}

\section{Rephrasing M{\oe}glin-Waldspurger algorithm}\label{sec lemma MW}

In this section, we introduce certain notation and a useful observation (Lemma \ref{lem M-W index} below) for M{\oe}glin-Waldspurger algorithm. Then we rephrase Algorithm \ref{algo M-W} in Corollary \ref{cor char of Delta}.

In the following discussion, we inherit the notation in Algorithm \ref{algo M-W}. Let $\beta$ be a multi-segment. We denote $\beta^0:=\beta$ and $\beta^i:= \beta^{i-1} \setminus M(\beta^{i-1})$, so that
\[ \widetilde{\beta}=  \{M(\beta^{l})\}_{l=0}^{i-1}+ \widetilde{\beta^i}. \]

Write $\beta=\{\Delta_j\}_{j \in J}$, and $\beta^{i}=\{\Delta_{j}^{i}\}_{j \in J^{i}}$. For $i >0$, we fix an injection $J^{i} \hookrightarrow J^{i-1}$, which identifies $J^{i}$ as a subset of $J^{i-1}$, with the following conditions.
\begin{enumerate}
    \item [$\oldbullet$] $\Delta_{j}^{i} \subseteq \Delta_j^{i-1}$, and 
    \item [$\oldbullet$] if $\Delta_{j}^{i-1} \neq \Delta_j^{i}$, then $\Delta_{j}^{i}= (\Delta_{j}^{i-1})^{-}$.
\end{enumerate}
In this way, we identify each $J^i$ as a subset of $J=J^0$. Define
\[ K^{i}:= \{ j \in J^i \ | \ \Delta_{j}^{i} \neq \Delta_{j}^{i+1}\}.\]
Then
\[ M(\beta^{i})= \{ e(\Delta_{j}^{i})\ | \ j \in K^i\}.\]
Write $M(\beta^{i})=[m^i, e^i]$. For $m^i \leq l \leq e^i$, let $k_l^i$ be the unique index in $K^i$ such that $e(\Delta_{k_l^i}^i)= l$.

With above notation, Algorithm \ref{algo M-W} can be rephrased as follows.

\begin{lemma}\label{lem char of Delta}
    With above notation, the following properties uniquely characterize $M(\beta^i)=[m^i,e^i]$ and $\{\Delta_{k_l^i}^i\}_{m^i \leq l \leq e^i}$.
    \begin{enumerate}
        \item [(1)] $ e^{i}= \max\{ e(\Delta_j^i)\ |\ j \in J^i \}. $
        \item [(2)] $  b(\Delta_{k_{e^i}^i}^i)= \max\{ b(\Delta_j^i) \ | \ j \in J^i,\  e(\Delta_j^i)=e^i    \}$.
        \item [(3)] $  b( \Delta_{k_l^i}^i)= \max\{ b(\Delta_j^i) \ | \ j \in J^i,\  e(\Delta_j^i)=l,\ b(\Delta_j^i)< b(\Delta_{k_{l+1}^i}^i)    \}$
        \item [(4)] $ \{ j \in J^i \ | \ e(\Delta_j^i)=m^i-1,\ b(\Delta_j^i)< b(\Delta_{k_{m^i}^i}^i) \}= \emptyset.$
    \end{enumerate}
\end{lemma}

Finally, let $e:= e^0= \max\{j \in J\ | \ e(\Delta_j)\}$ and set
\[ t:= \max(\{i \in \Z_{\geq 0}\ | \ e^i=e(M(\beta^i))=e \})+1.\]
Note that $t= \#\{ j \in J \ | \ e(\Delta_j)=e\}$. When there are more than one multi-segments involved in the argument, we write $J^i=J^i(\beta), K^{i}=K^{i}(\beta)$, $e=e(\beta)$ and $t=t(\beta)$ to specify the multi-segment $\beta$.

We give an example to demonstrate these notations.

\begin{exmp}
    Consider $\beta= \{\Delta_{j}\}_{j\in J}$, where $J=\{1,2,\dots, 8\}$, and
    \begin{align*}
        \Delta_1=[2,2], \ \Delta_{2}=[\ \ 0,1],&\ \Delta_{3}=[-2,0],\ \Delta_{4}=[-3,-1],\\
        \Delta_5=[1,2], \ \Delta_{6}=[-1,1],&\ \Delta_{7}=[-2,0],\ \Delta_{8}=[-2,1].
    \end{align*}
    We have $e(\beta)=2$ and $t(\beta)=2$. Applying Algorithm \ref{algo M-W} once, we obtain $M(\beta)=[m^0,e^0]=[-1,2]$, and $\beta^1= \beta \setminus M(\beta)$ is obtained from $\beta$ by replacing $\{\Delta_j\}_{j=1}^4$ by $\{\Delta_j^{-}\}_{j=1}^4$. Thus we can choose $K^0=\{1,2,3,4\}$ and $(k_{2}^0,k_{1}^0,k_{0}^0,k_{-1}^0)= (1,2,3,4)$. Note that one can also choose $K^0=\{1,2,7,4\}$ since $\Delta_{3}=\Delta_7$. 

    With the choice $K^0=\{1,2,3,4\}$, we obtain $J^1=\{2,3,\dots, 8\} \subseteq J$ and 
    \[ \beta^1= \{\Delta^1_j\}_{j=2}^8= \{\Delta_j^-\}_{j=2}^4 \sqcup \{\Delta_{j}\}_{j=5}^8.\]
    Apply Algorithm \ref{algo M-W} again, we obtain $M(\beta^1)=[m^1,e^1]=[0,2]$ and the only choice of $K^1$ is $\{ 5,6,7 \}$ with $(k_{2}^1,k_1^1,k_0^{1})=(5,6,7)$. We obtain $J^2=\{2,3,\dots,8\} \subseteq J^1 \subseteq J$ and 
    \[ \beta^2= \{ \Delta_j^2 \}_{j=2}^8= \{\Delta_j^-\}_{j=2}^4 \sqcup \{\Delta_{j}^-\}_{j=5}^7 \sqcup\{\Delta_8\}. \]
\end{exmp}

Observe that in above example, $K^0 \cap K^1= \emptyset$, or in other words, the index sets $\{K^i(\beta)\}_{i=0}^{t(\beta)-1}$ are mutually disjoint. One can also check this on the multi-segment in Example \ref{exmp lem}, which is slightly more complicated. We prove this interesting phenomenon for all multi-segments along with other properties in the following lemma.

\begin{lemma}\label{lem M-W index}
With the notation developed in this section, the followings hold for any multi-segment $\beta$.
\begin{enumerate}
    \item [(a)] $m^0\leq  m^1 \leq \cdots \leq  m^{t-1}$.
    \item [(b)] For any $0 \leq i \leq t-1$ and $m^i \leq l \leq e$, we have containment
  \[ \Delta_{k_{l}^0}^{0} \subseteq \Delta_{k_{l}^1}^{1} \subseteq \cdots \subseteq \Delta_{k_l^i}^{i}. \]
    \item [(c)]The sets $ \{K^i\}_{i=0}^{t-1}$ are mutually disjoint.
    \item [(d)] For $0 \leq i \leq t-1$  and $m^i \leq l \leq e$, we have 
    \[ \Delta_{k_l^i}=\Delta_{k_l^i}^0= \Delta_{k_l^i}^1 = \cdots =\Delta_{k_l^i}^i \supsetneq (\Delta_{k_l^i}^i)^{-} = \Delta_{k_l^i}^{i+1}= \cdots =\Delta_{k_l^i}^{t-1}. \]
\end{enumerate}
\end{lemma}
\begin{proof}
First, observe that Part (a) implies the indices $k_{l}^0,\dots, k_{l}^{i-1}$ in Part (b) is well-defined. Also, Part (d) is a direct consequence of Part (c). To prove Parts (a), (b), (c), we apply induction on $t=t(\beta)$. If $t=1$, the conclusions trivially hold. We assume $t>1$ from now on.

We claim that for any $l \in \{ e, e-1,\dots, m^{t-1} \}$, the followings hold for any $0 \leq i< t-1$.
\begin{enumerate}
    \item [(i)] $m^i\leq l$.
    \item [(ii)] $\Delta_{k_{l}^i}^{i} \subseteq \Delta_{k_{l}^{t-1}}^{t-1}$.
    \item [(iii)] $ k_{l}^{t-1}$ is not in $K^{i}$.
\end{enumerate}

Assume the Claims are verified. Set $J':= J \setminus K^{t-1}$ and $\beta':= \{\Delta_j\}_{j \in J'}$. Then
\[ t(\beta')=\#\{j \in J'\ | \ e(\Delta_j)=e\}= t-1.\]
Also, for any $0 \leq i<t-1$, Claim (iii) implies $K^{t-1} \cap K^i= \emptyset$, and hence Claim (ii) implies for $m^i \leq l \leq e^i$,
\[ \Delta_{k_l^i}^i \subseteq \Delta_{k_{l}^{t-1}}^{t-1}=\Delta_{k_{l}^{t-1}}^{i}. \]
This gives $b(\Delta_{k_l^i}^i ) \geq b(\Delta_{k_{l}^{t-1}}^{i} )$. From this inequality, for $0 \leq i < t-1$ and  $m^i \leq l \leq e^i$, one can inductively check that $[m^i, e^i]$ and $\Delta_{k_l^i}^i \in (\beta')^{i}$, satisfy the properties in Lemma \ref{lem char of Delta}. Therefore, we can choose $K^{i}(\beta')= K^i(\beta)$. Thus the induction hypothesis for $\beta'$ implies $m^0 \leq \cdots \leq m^{t-2}$, Part (b) holds for $0 \leq i\leq t-2$, and $\{K^i\}_{i=0}^{t-2}$ are mutually disjoint. Together with the Claims, this verifies all of the desired conclusions for $\beta$.

Now we prove the Claims by applying induction on $l$. When $l=e$, Claims (i) trivially holds. Claim (ii) follows from Lemma \ref{lem char of Delta}(2). For claim (iii), write $\Delta_{k_e^{t-1}}^0=[x,y]$. Then $\Delta_{k_e^{t-1}}^t=[x, y-s]$ where $s= \#\{0 \leq i<t-1\ | \ k_{e}^{t-1} \in K^i\}$. Since $e(\Delta_{k_{e}^{t-1}}^{t-1})=e$ and $y \leq e$ by definition, we obtain
\[e= y-s \leq y\leq e,  \]
which implies $s=0$. This completes the verification of the claims for $l=e$.

Suppose $e> r \geq m^{t-1}$ and the Claims are already verified for $l=r+1$. We are going to verify the Claims for $l=r.$  

First, for any $0 \leq i < t-1$, Claim (ii) for $l=r+1$ gives $\Delta_{k_{r+1}^i}^i \subseteq \Delta_{k_{r+1}^{t-1}}^{t-1}$. Combining with Lemma \ref{lem char of Delta}(3), we obtain
\begin{align}\label{eq proof of claim precede}
     b(\Delta_{k_{r}^{t-1}}^{t-1}) < b(\Delta_{k_{r+1}^{t-1}}^{t-1}) \leq b(\Delta_{k_{r+1}^i}^i).
\end{align}
On the other hand, Claim (iii) for $l=r+1$ implies $k_{r+1}^{t-1} \neq k_{r}^j$ for any $0 \leq j <t-1$, and hence 
\begin{align}\label{eq proof of claim chain of eq}
    \Delta_{k_{r}^{t-1}}^0= \Delta_{k_{r}^{t-1}}^1= \cdots = \Delta_{k_{r}^{t-1}}^i = \cdots =\Delta_{k_{r}^{t-1}}^{t-1}.
\end{align}
 As a consequence, the set 
\[ \{ j \in J^i\ | \ e(\Delta_j^i)=r,\ b(\Delta_j^i)< b(\Delta_{k_{r+1}^i}^i) \}\]
is non-empty since it contains $k_{r}^{t-1}$. We conclude that $m^i\leq r$ by Lemma \ref{lem char of Delta}(4). This proves Claim (i) for $l=r$. 

Next, for any $0 \leq i \leq t-1$, \eqref{eq proof of claim precede} and \eqref{eq proof of claim chain of eq} give
\[ e(\Delta_{k_{r}^{t-1}}^i)=e(\Delta_{k_{r}^{t-1}}^{t-1})=r,\ b(\Delta_{k_{r}^{t-1}}^i)=b(\Delta_{k_{r}^{t-1}}^{t-1}) < b(\Delta_{k_{r+1}^i}^i). \]
Therefore, $k_{r}^{t-1}$ is in the set
\[ \{ j \in J^i\ |\ e(\Delta_{j}^i)=r,\ b(\Delta_{j}^{i}) < b(\Delta_{k_{r+1}^i}^i) \}.\]
We conclude that $ b(\Delta_{k_{r}^{t-1}}^i) \leq b(\Delta_{k_{r}^i}^i)$ by Lemma \ref{lem char of Delta}(3). This verifies Claim (ii) for $l=r$. 

Finally for Claim (iii), suppose the contrary that $k_{r}^{t-1} \in K^i$ for some $0 \leq i <t$ and take maximal such $i$. We must have 
\[\Delta_{k_r^{t-1}}^{t-1}=\Delta_{k_r^{t-1}}^{t-2}= \cdots=\Delta_{k_r^{t-1}}^{i+1}=(\Delta_{k_r^{t-1}}^{i})^{-}, \]
and hence $k_{r}^{t-1}= k_{r+1}^i$. However, Claim (ii) for $l=r+1$ gives
\[ \Delta_{k_{r}^{t-1}}^{t-1}= \Delta_{k_{r+1}^i}^{i+1}= (\Delta_{k_{r+1}^i}^{i})^{-}\subsetneq \Delta_{k_{r+1}^{i}}^{i} \subseteq   \Delta_{k_{r+1}^{t-1}}^{t-1}.\]
In particular, $b(\Delta_{k_{r}^{t-1}}^{t-1} ) \geq b(\Delta_{k_{r+1}^{t-1}}^{t-1})$, which contradicts to Lemma \ref{lem char of Delta}(3). This completes the verification of Claim (iii) for $l=r$ and the proof of the lemma. 
\end{proof}

As a corollary, we can improve the statement of Lemma \ref{lem char of Delta} when $0 \leq i \leq t-1$ as follows.

\begin{cor}\label{cor char of Delta}
Let $\{K^i\}_{i=0}^{t(\beta)-1}$ be a collection of mutually disjoint subsets of $J(\beta)$. The followings are equivalent.
\begin{enumerate}
    \item [(a)] $K^i$ can be labeled as $K^i=\{k_{l}^i\}_{m^i \leq l \leq e(\beta)}$ such that the followings hold.
    \begin{enumerate}
\item [(1)] $e(\Delta_{k_l^i})=l$.
    \item [(2)]$  b(\Delta_{k_{e}^i})= \max\{ b(\Delta_j) \ | \ j \in J\setminus (\sqcup_{0 \leq r <i} K^r ),\  e(\Delta_j)=e    \}$.
        \item [(3)] $  b( \Delta_{k_l^i})= \max\{ b(\Delta_j) \ | \ j \in J\setminus (\sqcup_{0 \leq r <i} K^r),\  e(\Delta_j)=l,\ b(\Delta_j)< b(\Delta_{k_{l+1}^i})    \}$
        \item [(4)] $ \{ j \in J\setminus (\sqcup_{0 \leq r <i} K^r) \ | \ e(\Delta_j)=m^i-1,\ b(\Delta_j)< b(\Delta_{k_{m^i}^i}) \}= \emptyset.$
\end{enumerate}
\item [(b)] The multi-segment $\beta^{i}$ can be obtained from $\beta^{i-1}$ by replacing $\{\Delta_{k^i}\}_{k^i \in K^i}$ with $\{\Delta_{k^i}^{-}\}_{k^i \in K^i}$.
\end{enumerate}
\end{cor}

\begin{proof}
    The description of Parts (a) and (b) both determine the collection of segments $\{\Delta_{k^i_l}\}_{k^i \in K^i}$. Thus it suffices to examine the index sets $\{K^i\}_{i=0}^{t-1}$ chosen in the beginning of this section, where Part (b) holds, satisfies all of the conditions in Part (a). 
    
    Condition (1) holds by definition. Conditions (2) holds since $e(\Delta_j^i)=e$ if and only if $\Delta_j^i=\Delta_j$ and $j \not\in \sqcup_{0\leq r <i}K^r$. For Condition (3), observe that $\Delta_{j}^i= \Delta_j$ for $j \in J \setminus \sqcup_{0\leq r <i}K^r$,  and hence Lemma \ref{lem char of Delta}(3) implies
    \[ b( \Delta_{k_l^i})\geq  \max\{ b(\Delta_j) \ | \ j \in J\setminus (\sqcup_{0 \leq r <i} K^r),\  e(\Delta_j)=l,\ b(\Delta_j)< b(\Delta_{k_{l+1}^i})    \}.\]
    On the other hand, Lemma \ref{lem M-W index}(c) shows that $ k_{l}^i\in J\setminus (\sqcup_{0 \leq r <i} K^r)$, so the equality holds. By the same observation, Lemma \ref{lem char of Delta}(4) implies Condition (4). This completes the proof of the corollary.
\end{proof}

\section{\texorpdfstring{Proof of Lemma \ref{lem main intro}}{}}\label{sec proof}

\subsection{Rephrasing Lemma \ref{lem main intro} by multi-segments}\label{sec rephrase}
In this subsection, we rephrase Lemma \ref{lem main intro} using the notation of multi-segments recalled in \S \ref{sec preliminary}.

First, we define multi-segments of \emph{Arthur type}.

\begin{defn}
We say a multi-segment $\alpha$ is of Arthur type if $\alpha= \delta_{\phi}$ for some $L$-parameter $\phi$ of Arthur type. The following is an equivalent but more explicit description. 
\begin{enumerate}
    \item [1.]For an irreducible local Arthur parameter $\rho \otimes \Sym^{d} \otimes \Sym^{a}$, we associate a multi-segment
\[ \delta_{d,a}:=\left\{ \left[\half{a-d}, \half{a+d}\right]_{\rho}, \left[\half{a-d}-1, \half{a+d}-1\right]_{\rho}, \dots,    \left[\half{-a-d}, \half{-a+d}\right]_{\rho}\right\}.\]
\item [2.]For a local Arthur parameter of the form
\begin{align}\label{eq parameter}
    \psi= \bigoplus_{\rho} \bigoplus_{i\in I_{\rho}} \rho \otimes \Sym^{d_i} \otimes \Sym^{a_i}
\end{align}
we associate a multi-segment \[\delta_{\psi}:=\sum_{\rho} \sum_{i \in I_{\rho}} \delta_{d_i,a_i}. \]
\item [3.]We say a multi-segment $\alpha$ is of \emph{Arthur type} if $\alpha= \delta_{\psi}$ for some $\psi$ of the form \eqref{eq parameter}.  
\end{enumerate}
\end{defn}

With this definition, we may rephrase Lemma \ref{lem main intro} as follows.
\begin{lemma}\label{lem main segment}
    Let $\alpha $ be a multi-segment of Arthur type. If $\beta$ is a multi-segment such that $\beta \geq \alpha$ and $\widetilde{\beta} \geq \widetilde{\alpha}$, then $\beta=\alpha$.
\end{lemma}

\subsection{Preparation for reduction}\label{sec reduction}
The goal of this subsection is to prove Proposition \ref{prop main} below, which is the key of the reduction argument in the proof of Lemma \ref{lem main intro}. The motivation of the formulation of Proposition \ref{prop main} is discussed in the introduction. Note that some of the statement in this subsection is similar to those in \cite{Rid23}, but the proofs are different.




\begin{prop}\label{prop main}
Fix $\rho \in \mathcal{C}_{\textrm{unit}}$. Suppose $\alpha$ is of Arthur type with $\alpha=\delta_{\psi}$, where $\psi$ is of the form
\[ \psi= \bigoplus_{i\in I_{\rho}} \rho \otimes \Sym^{d_i} \otimes \Sym^{a_i}.  \]
Denote
\begin{align*}
    a+d &:= \max\{ a_i+d_i \ | \ i \in I_{\rho}\},\\
    d&:= \min \{ d_i \ | \ i \in I_{\rho},\ a_i+d_i=a+d  \}.
\end{align*}
Then for any multi-segment $\beta$ such that $\beta \geq \alpha$ and $\widetilde{\beta} \geq \widetilde{\alpha}$, the following holds.
\begin{enumerate}
    \item [(i)]The multi-segment $\beta$ must contain a copy of $\delta_{d,a}$.
    \item [(ii)] If we define $\alpha^{-}$ and $\beta^{-}$ by removing a copy of $\delta_{d,a}$ from $\alpha$ and $\beta$ respectively, then $\alpha^{-}= \delta_{\psi^{-}}$ where 
\[ \psi^{-}:= \psi- \rho \otimes \Sym^{d}\otimes \Sym^{a},\]
and we have $ \beta^{-} \geq \alpha^{-}$ and $\widetilde{\beta^{-}} \geq \widetilde{\alpha^{-}}$. 
\end{enumerate}
\end{prop}
For simplicity, we shall omit the subscript $\rho$ of $[b,e]_{\rho}$ in the rest of this subsection.

We first show that the choice of $a+d$ and $d$ in Proposition \ref{prop main} gives the following bounds. The definition of $\beta^i$ and other notation in the proof can be found in \S \ref{sec lemma MW}.

\begin{lemma}\label{lem a-d}
Under the setting of Proposition \ref{prop main}, for any $i \geq 0$ and $[x,y] \in \beta^i$, the following holds.
\begin{enumerate}
    \item [(a)] $x \geq \half{-a-d}$.
    \item [(b)] $y \leq \half{a+d}$. If the equality holds, then $[b,e]\in \beta^l$ for any $0 \leq l \leq i$ and $x \leq \half{a-d}$.
\end{enumerate}
\end{lemma}
\begin{proof}
For arbitrary multi-segments $\alpha, \beta$, we have the following observations.
\begin{enumerate}
    \item [$\oldbullet$] If $\beta\geq \alpha$, then $\supp(\beta)=\supp(\alpha)$.
    \item [$\oldbullet$] For $i \geq 0$, $\supp(\beta^i) = \supp(M(\beta^i))+ \supp(\beta^{i+1})$. In particular, we have  $\supp(\beta) \supseteq \supp(\beta^i)$.
\end{enumerate}
Now we return to the setting of Proposition \ref{prop main}.
The choice of $a+d$ gives
\begin{align}
  \label{eq a+d}  \half{a+d}&= \max\{ x \in \R \ | \ \rho|\cdot|^{x} \in \supp(\alpha)=\supp(\beta) \}, \\
  \nonumber  \half{-a-d}&= \min\{ x \in \R \ | \ \rho|\cdot|^{x} \in \supp(\alpha)=\supp(\beta) \},
\end{align}
which shows Part (a) and the first part of Part (b) by the observations above. 

Next, we show the second part of Part (b) with the notation developed in \S \ref{sec lemma MW} for $\beta$. Take a $j \in J^i$ such that  $\Delta_{j}^i=[x,y]$. We have 
\[ \Delta_{j}^0 \supseteq \Delta_{j}^1 \supseteq \cdots \supseteq \Delta_j^i=[x,y].\]
Moreover, if any of the inclusion is strict, then $e(\Delta_j^0)>y$. If $y=\half{a+d}$, this can not happen by \eqref{eq a+d}. Therefore, we conclude that $[x,y] =\Delta_{j}^l \in\beta^l $ for any $0 \leq l \leq i$.


Finally, we verify the property that if $[x,\half{a+d}] \in \beta$ then $x \leq \half{a-d}$. If $\beta=\alpha$, then 
\[ \left[x, \half{a+d} \right] \in \left\{ \left[\half{a+d}-d_i, \half{a+d} \right]  \ | \ i \in I_{\rho}, \ a_i+d_i= a+d \right\},  \]
and hence $ x \leq \half{a-d}$ by the definition of $d$. In general, $\beta$ is obtained from $\alpha$ by performing a sequence of elementary operations. By \eqref{eq a+d}, the desired property is preserved under each elementary operation, and hence $\beta$ also has this property. This completes the proof of the lemma.
\end{proof}

For Part (ii) in Proposition \ref{prop main}, we need the following computation, which works for a general multi-segment $\beta$. We obtain the same conclusion as \cite[Lemma 3.18]{Rid23} but with slightly weaker conditions.

\begin{lemma}\label{lemma M-W}
Suppose a multi-segment $\beta$ contains a copy of
\[ \delta_{b,e,s}:=\{ [b,e],[b-1,e-1],\dots, [b-s,e-s]   \}, \]
and any segment $[x,y]\in \beta$ satisfies the following assumptions.
\begin{enumerate}
    \item [(i)] $x \geq b-s$.
    \item [(ii)]  $ y \leq e$. 
\end{enumerate}
Define $\beta^{-}$ by removing a copy of $\delta_{b,e,s}$ from $\beta$. Then we have 
\[ \widetilde{\beta}= \widetilde{\beta^{-}}+ \widetilde{\delta}_{b,e,s}, \]
where
\[ \widetilde{\delta}_{b,e,s}= \{ [e-s,e], [e-s-1,e-1],\dots, [b-s,b] \}. \]
\end{lemma}

We give an example to illustrate the reduction process of the proof.
\begin{exmp}\label{exmp lem}
   Let $\beta= \sum_{i=0}^4\{\Delta_{k^i}\}_{k^i \in K^i}$, where
   \begin{align}\label{eq seg exmp}
   \begin{split}
       \{\Delta_{k^0}\}_{k^0 \in K^0}&= \{[2,2],[\ \ 1,1],[\ \ 0, 0],[-1,-1],[-2,-2]\},\\
       \{\Delta_{k^1}\}_{k^1 \in K^1}&= \{[1,2],[-1,1],\bf{[-2,0]},\bf{[-3,-1]}\},\\
       \{\Delta_{k^2}\}_{k^2 \in K^2}&= \{[0,2],{\bf{[-1,1]}},[-3,0]\},\\
       \{\Delta_{k^3}\}_{k^3 \in K^3}&= \{{\bf{[0,2]}},[-3,1]\},\\
       \{\Delta_{k^4}\}_{k^4 \in K^4}&= \{[-2,2],\}.
   \end{split}
   \end{align}
   Note that $\beta$ contains $\delta_{0,2,3}=\{ [0,2],[-1,1],[-2,0],[-3,-1]\},$ and the pair $(\beta, \delta_{0,2,3} )$ satisfies Assumptions (i), (ii) in the lemma. Also, the subsets $\{K^i\}_{i=0}^4$ of $J(\beta)= \sqcup_{i=0}^4 K^i$ satisfy all of the conditions in Corollary \ref{cor char of Delta}(a), where we let $k_l^i$ denote to the index in $K^i$ such that $e(\Delta_{k_l^i})=l$. Therefore, $\beta^{4}=\sum_{i=0}^4\{\Delta_{k^i}^-\}_{k^i \in K^i}$, and
   \begin{align*}
       \widetilde{\beta}&= \{M(\beta^i)\}_{i=0}^4+ \widetilde{\beta^{4}}\\
       &= \{ [-2,2],[-1,2],[0,2],[1,2],[2,2] \}+ \widetilde{\beta^{4}}.
   \end{align*}
   
   For $-1=e-s \leq l \leq e= 2$, we let $r_l:=\max \{ 0 \leq r\leq 4 \ | \ \Delta_{k_l^r}= [e-s-l,e-l]  \}$. Thus $(r_2,r_1,r_0,r_{-1})=(3,2,1,1)$. The corresponding segments $\Delta_{k_{l}^{r_l}}$ are displayed in bold text in \eqref{eq seg exmp}.   

   We may construct the segment $\beta^{-}$ from $\beta$ by removing $\{\Delta_{k_{l}^{r_l}}\}_{-1 \leq l\leq 2}$. Let $J^{-}:= J\setminus \{k_l^{r_l}\}_{-1 \leq l \leq 2}$, which we identify with $J(\beta^{-})$. We define the mutually disjoint subsets $\{(K^i)^{-}\}_{i=0}^3$ of $J^-$ by removing $\{\Delta_{k_{l}^{r_l}}\}_{-1 \leq l \leq 2}$ in \eqref{eq seg exmp} and then ``push" the segments above them upward accordingly. That is, we define
    \begin{align*}
       \{\Delta_{(k^0)^{-}}\}_{(k^0)^{-} \in (K^0)^{-}}&= \{[2,2],[\ \ 1,1],[\ \ 0, 0],[-1,-1],[-2,-2]\},\\
       \{\Delta_{(k^1)^{-}}\}_{(k^1)^{-} \in (K^1)^{-}}&= \{[1,2],[-1,1],[-3,0]\},\\
       \{\Delta_{(k^2)^{-}}\}_{(k^2)^{-} \in (K^2)^{-}}&= \{[0,2],[-3,1]\},\\
       \{\Delta_{(k^3)^{-}}\}_{(k^3)^{-} \in (K^3)^{-}}&= \{[-2,2]\}.
   \end{align*}
   It follows that the mutually disjoint subsets $\{(K^i)^-\}_{i=0}^3$ of $J^-$ satisfies all of the conditions in Corollary \ref{cor char of Delta}(a). Therefore, $(\beta^{-})^3=\sum_{i=0}^3\{\Delta_{(k^i)^-}^{-}\}_{(k^i)^- \in (K^i)^-}$, and 
    \begin{align*}
       \widetilde{\beta^-}&= \{M((\beta^-)^i)\}_{i=0}^3+ \widetilde{(\beta^{-})^3}\\
       &= \{ [-2,2],[0,2],[1,2],[2,2] \}+ \widetilde{(\beta^{-})^3}.
   \end{align*}
   Now observe that $\beta^4$ contains a copy of $\{\Delta_{k_l^{r_l}}^-\}_{-1 \leq l \leq 3}=\delta_{0,1,3}$ and the pair $(\beta^4, \delta_{0,1,3}) $ satisfies Assumptions (i), (ii) in the lemma. Moreover, $(\beta^4)^{-}$, which is obtained from $\beta^4$ by removing a copy of $\delta_{0,1,3}$, is exactly $(\beta^-)^3$, and $\{M(\beta^i)\}_{i=0}^4= \{M((\beta^-)^i)\}_{i=0}^3 \sqcup \{[2,-1]\}$. We conclude that the lemma holds for the pair $(\beta, \delta_{0,2,3})$ if it holds for the pair $(\beta^{4}, \delta_{0,1,3})$.
\end{exmp}

Now we apply the reduction process in above specific example to prove the lemma for general cases.

\begin{proof}[proof of Lemma \ref{lemma M-W}]
We use the notation developed in \S \ref{sec lemma MW} for $\beta$, which we recall briefly as follows.

\begin{enumerate}
    \item [$\oldbullet$]$\beta= \{\Delta_{j}\}_{j \in J(\beta)}$.
    \item [$\oldbullet$]$e(\beta):= \max\{e(\Delta_{j})\}_{j \in J(\beta)}$.
    \item [$\oldbullet$]$t(\beta):=\#\{j \in J(\beta)\ | \ e(\Delta_j)=e(\beta)\}.$
   \item [$\oldbullet$] $\{K^i(\beta)\}_{i=0}^{t-1}$ is a collection of (mutually disjoint) subsets of $J(\beta)$ satisfying all conditions in Part (a) of Corollary \ref{cor char of Delta}.  
\end{enumerate}
Note that Assumption (ii) implies that the number $e$ matches $e(\beta)$. We shall simply write $J=J(\beta)$, $t=t(\beta)$, $K^i=K^i(\beta)$ in the following discussion.
    
     We apply induction on $\ell:=e-b+1$ to prove the lemma. When $\ell=1$, it follows directly from Algorithm \ref{algo M-W} that $M(\beta)= \widetilde{\delta}_{b,e,s}$, $\beta^1=\beta^-$, and hence
     \begin{align*}
         \widetilde{\beta}&= M(\beta)+ \widetilde{\beta^1}\\
         &= \widetilde{\beta^{-}}+ \widetilde{\delta}_{b,e,s}.
     \end{align*}
    From now on, we assume $\ell>1$ and the conclusions are already established for $e-b+1=\ell-1$.

    First, we give the following observations.
    \begin{enumerate}
        \item [$\oldbullet$] $\Delta_{k_e^i}=[b,e]$ for some $0 \leq i \leq t-1$.
        \item [$\oldbullet$] 
       If $0 \leq r<s$, $\Delta_{k_{e-r}^i}=[b-r,e-r]$, and $\beta^i$ contains a $[b-r-1,e-r-1]$, then $\Delta_{k_{e-r-1}^i}=[b-r-1,e-r-1]$.
    \end{enumerate}
    As a consequence, $\bigsqcup_{i=0}^{t-1} \{\Delta_{k^i}\}_{k^i \in K^i}$ must contain a copy of $\delta_{b,e,s}$. Thus for each $e-s \leq l \leq e$, the set
    \[\{ 0 \leq r \leq t-1 \ | \ \Delta_{k_l^r}=[l-e+b,l]   \}\]
    is non-empty, and we define $r_l$ to be the maximum of this set.
    
    Next, we show that
      \begin{align}\label{eq r increasing}
          r_{e-s} \leq r_{e-s+1} \leq \cdots \leq r_{e}.
      \end{align}
    Indeed, for $e-s < l \leq e$, Parts (b) and (d) of
    Lemma \ref{lem M-W index}  and the definition of $r_l$ imply that for any $ r_l <i \leq t-1$,
    \[ [l-e+b,l]= \Delta_{k_{l}^{r_l}} \subsetneq \Delta_{k_{l}^{i}}.\]
    Thus Corollary \ref{cor char of Delta}(a)(3) shows that if $l-1 \geq m^i$, then 
    \[b(\Delta_{k_{l-1}^{i}}) < b(\Delta_{k_{l}^{i}}) \leq l-e+b-1,\]
    and hence $\Delta_{k_{l-1}^{i}} \neq [(l-1)-e+b,l-1]$. This shows that $r_{l-1} \leq r_{l}$, and hence \eqref{eq r increasing} holds.

     Now set $J^{-}:= J \setminus \{ k^{r_l}_l\}_{l=e-s}^e$. We have $\beta^{-}=\{\Delta_{j}\}_{j \in J^{-}}$. For $0 \leq i \leq t-2$, define mutually disjoint subsets $(K^i)^{-}$ of $J^{-}$ as follows.
       \begin{align*}
        (K^i)^{-}:=\begin{cases}
            K^i & \textrm{ if }i<r_{e-s},\\
            \{ k_{e}^i,\dots, k_{l+1}^i\} \sqcup \{k_l^{i+1},\dots, k_{m^{i+1}}^{i+1}\} & \textrm{ if } r_{l} \leq  i< r_{l+1}\\
            K^{i+1} & \textrm{ if } i\geq r_{e}. 
        \end{cases}
    \end{align*}
    We claim that the collection of index sets $\{ (K^i)^-\}_{i=0}^{t-2}$ satisfies all of the conditions in Corollary \ref{cor char of Delta}(a), and hence we can compute $\widetilde{\beta^-}$ using these index sets as described in Part (b) of the same corollary. 
    
   Before verifying the claim, we demonstrate that the claim implies the desired conclusion. Define 
         \begin{align*}
         K^-:=\bigsqcup_{i=0}^{t-2}(K^i)^{-},\ K:=\bigsqcup_{i=0}^{t-1}K^i.
     \end{align*}
    Observe that  $ b(\Delta_{k_{e-s}^{r_{e-s}}})= b-s$, and hence by Assumption (i), $M(\beta^{r_{e-s}})=[ e-s, e]$. This implies
    \begin{align}\label{eq of K}
        K^- \sqcup \{ k^{r_l}_l\}_{l=e-s}^e=K
    \end{align}
    As a consequence of above equality, 
    \begin{align*}
        \beta^{t}= (\beta^{-})^{t-1} \sqcup \delta_{b,e-1,s},
    \end{align*}
   and hence $\beta^{t}$ contains a copy of $\delta_{b,e-1,s}$. Also, $\beta^t$
    satisfies Assumptions (i), (ii) with $e$ replaced by $e-1$ according to its construction.  Thus the induction hypothesis gives
    \begin{align}\label{eq beta^t}
        \widetilde{\beta^t}= \widetilde{(\beta^-)^{t-1}} \sqcup \widetilde{\delta}_{b,e-1,s}.
    \end{align}
        Another consequence of the equality \eqref{eq of K} is that
     \begin{align}\label{eq of M}
         \{ M((\beta^{-})^{i})\}_{i=0}^{t-2} \sqcup \{ [e-s,e]\} = \{ M(\beta^{i})\}_{i=0}^{t-1}.
     \end{align}
Indeed, observe that as multi-sets over $\mathcal{C}$, the following equality holds 
    \[ \sum_{i=0}^{t-1} M(\beta^{i}) = \sum_{k \in K} \{e(\Delta_k)\},  \]
     and the left hand side above uniquely determines the collection $\{M(\beta^i)\}_{i=0}^{t-1}$. The same argument works for $\beta^{-}$. Therefore, \eqref{eq of K} implies \eqref{eq of M}.

     Combining \eqref{eq beta^t} and \eqref{eq of M}, we obtain
     \begin{align*}
         \widetilde{\beta}&=\{M(\beta^{i})\}_{i=0}^{t-1}+ \widetilde{\beta^t}\\
         &= \{ M((\beta^{-})^{i})\}_{i=0}^{t-2} \sqcup \{ [e-s,e]\} + \widetilde{(\beta^-)^{t-1}} \sqcup \widetilde{\delta}_{b,e-1,s}\\
         &= \widetilde{\beta^-} + \widetilde{\delta}_{e,b-1,s}.
     \end{align*}
This gives the desired conclusion of the lemma.

Now we prove the claim. Write $(K^{i})^-= \{ (k_l^{i})^- \}_{(m^i)^{-}\leq l \leq e}$, where $e(\Delta_{(k_l^{i})^-})=l$. In other words,
\[ (m^i)^{-}= \begin{cases}
    m^i & \text{ if }i < r_{e-s},\\
    m^{i+1} & \text{ if }i \geq  r_{e-s},
\end{cases}\  (k_l^i)^{-}= \begin{cases}
    k_l^i & \text{ if }i< r_l,\\
    k_{l}^{i+1} & \text{ if } l\geq (m^{i})^-\text{ and } i\geq  r_l.
 \end{cases} \]
First, Conditions (1) and (2) of Corollary \ref{cor char of Delta}(a) for $\{(K^i)^-\}_{i=0}^{t-2}$ hold by construction.

Next, we verify Condition (3) for each $\Delta_{(k_{l}^i)^-}$, where $0 \leq i \leq t-2$ and $e> l \geq (m^i)^{-}$. We separate into three cases:  $  r_{l+1} \geq r_{l}>i$,  $r_{l+1} > i \geq r_l$, and $i\geq r_{l+1} \geq r_l$. The key observation is the following equality.
\begin{align}\label{eq of K^-}
\begin{split}
    \{j \in J \setminus (\sqcup_{0 \leq r < i} K^r \ | \ e(\Delta_j)=l\}
       =&\{j \in J^{-} \setminus (\sqcup_{0 \leq r < i} (K^r)^{-}) \ | \ e(\Delta_j)=l\} \\
       &\sqcup\{k_l^{\max(i,r_l)}\}.
\end{split}   
\end{align}

\textbf{Case 1.} Suppose $  r_{l+1} \geq r_{l}>i$. Then $(k_{l}^i)^{-}= k_l^{i}$, $(k_{l+1}^i)^{-}= k_{l+1}^{i}$, and $b(\Delta_{k_{l}^i})\geq b(\Delta_{k_{l}^{r_l}})$ by Lemma \ref{lem M-W index}(b) and (d). Thus Corollary \ref{cor char of Delta}(a)(3) for $\Delta_{k_l^i}$ and \eqref{eq of K^-} give
\begin{align*}
    b(\Delta_{(k_{l}^i)^-})&=b(\Delta_{k_{l}^i})\\
    &= \max\{ b(\Delta_j)\ | \ j \in J\setminus (\sqcup_{0 \leq r <i}K^r),\ e(\Delta_j)=l,\ b(\Delta_{j})< b(\Delta_{k_{l+1}^i})\}\\
    &= \max\huge{(}\{ b(\Delta_j)\ | \ j \in J^-\setminus  (\sqcup_{0 \leq r <i}(K^r)^-),\ e(\Delta_j)=l,\ b(\Delta_{j})< b(\Delta_{k_{l+1}^i})\} \\
    &\ \ \ \ \ \ \ \ \ \ \ \ \ \ \sqcup\{ b(\Delta_{k_l^{r_l}})\}\huge{)}\\
     &= \max\{ b(\Delta_j)\ | \ j \in J^-\setminus (\sqcup_{0 \leq r <i}(K^r)^-),\ e(\Delta_j)=l,\ b(\Delta_{j})< b(\Delta_{(k_{l+1}^i)^{-}})\}.
\end{align*}
This verifies Condition (3) in this case.

\textbf{Case 2.} Suppose $r_{l+1}>i\geq r_l$. Then $(k_l^i)^{-}= k_{l}^{i+1}$ and $(k_{l+1}^i)^-= k_{l+1}^i$. 

First, we check that for any $j^{-} \in  J^{-} (\setminus \sqcup_{0 \leq r \leq i-1} (K^r)^{-})$, if $e(\Delta_{j^-})=l$ and $b(\Delta_{j})< b(\Delta_{k_{l+1}^i}),$ then $b(\Delta_{j})< b(\Delta_{k_{l+1}^{i+1}})$. Indeed, take any $j^{-}$ satisfying above assumptions. Corollary \ref{cor char of Delta}(a)(3) for $\Delta_{k_{l}^i}$ and \eqref{eq of K^-} imply 
\begin{align*}
    b(\Delta_{k_{l}^i}) &=\max\{ b(\Delta_j)\ | \   j \in J \setminus (\sqcup_{0 \leq r \leq i-1} K^r), e(\Delta_j)=l, b(\Delta_{j})< b(\Delta_{k_{l+1}^i})\}\\
    &\geq b(\Delta_{j^{-}}).
\end{align*}
Then since $r_{l+1} \geq i+1 >i \geq r_{l}$, Lemma \ref{lem M-W index}(b) and (d) imply
\[ b(\Delta_{j^{-}}) \leq b(\Delta_{k_{l}^{i}})\leq b(\Delta_{k_{l}^{r_l}})=l-e+b<l+1-e+b = b(\Delta_{k_{l+1}^{r_{l+1}}})\leq b(\Delta_{k_{l+1}^{i+1}}).\]
In particular $b(\Delta_{j})< b(\Delta_{k_{l+1}^{i+1}})$ holds. As a consequence, we obtain
\begin{align*}
   & b(\Delta_{(k_l^i)^{-}})\\
    =&b(\Delta_{k_l^{i+1}})\\
    =& \max\{ b(\Delta_j)\ | \   j \in J \setminus (\sqcup_{0 \leq r < i+1} K^r), e(\Delta_j)=l, b(\Delta_{j})< b(\Delta_{k_{l+1}^{i+1}})\}\\
    =& \max\left(\{ b(\Delta_j)\ | \   j \in J \setminus (\sqcup_{0 \leq r < i} K^r), e(\Delta_j)=l, b(\Delta_{j})< b(\Delta_{k_{l+1}^{i+1}})\} \setminus \{ b(\Delta_{k_l^i})\} \right)\\
    =& \max\{ b(\Delta_j)\ | \   j \in J^- \setminus (\sqcup_{0 \leq r < i} (K^r)^-), e(\Delta_j)=l, b(\Delta_{j})< b(\Delta_{k_{l+1}^{i+1}})\}\\
    =& \max\{ b(\Delta_j)\ | \   j \in J^- \setminus (\sqcup_{0 \leq r < i} (K^r)^-), e(\Delta_j)=l, b(\Delta_{j})< b(\Delta_{k_{l+1}^{i}})\}\\
    =&\max\{ b(\Delta_j)\ | \   j \in J^- \setminus (\sqcup_{0 \leq r < i} (K^r)^-), e(\Delta_j)=l, b(\Delta_{j})< b(\Delta_{(k_{l+1}^{i})^{-}})\}.
\end{align*}
This verifies Condition (3) in this case.

\textbf{Case 3.} Suppose $i\geq r_{l+1}\geq r_l$. Then $(k_l^i)^{-}= k_{l}^{i+1}$ and $(k_{l+1}^i)^-= k_{l+1}^{i+1}$. Corollary \ref{cor char of Delta}(a)(3) for $\Delta_{k_l^{i+1}}$ and \eqref{eq of K^-} give
\begin{align*}
    b(\Delta_{(k_{l}^i)^-})&=b(\Delta_{k_{l}^{i+1}})\\
    &= \max\{ b(\Delta_j)\ | \ j \in J\setminus (\sqcup_{0 \leq r <i+1}K^r),\ e(\Delta_j)=l,\ b(\Delta_{j})< b(\Delta_{k_{l+1}^{i+1}})\}\\
    &= \max\huge{(}\{ b(\Delta_j)\ | \ j \in J^-\setminus ( \sqcup_{0 \leq r <i+1}(K^r)^-),\ e(\Delta_j)=l,\ b(\Delta_{j})< b(\Delta_{k_{l+1}^{i+1}})\} \\
    &\ \ \ \ \ \ \ \ \ \ \ \ \ \ \sqcup\{ b(\Delta_{k_l^{i+1}})\}\huge{)}\\
     &= \max\{ b(\Delta_j)\ | \ j \in J^-\setminus (\sqcup_{0 \leq r <i}(K^r)^-),\ e(\Delta_j)=l,\ b(\Delta_{j})< b(\Delta_{(k_{l+1}^i)^{-}})\}.
\end{align*}
This verifies Condition (3) in this case.

Finally, we verify Condition (4). We separate into three cases: $i <r_{e-s} $, $r_{l+1}> i \geq r_{l}$ for some $ e-s\leq l < e$, and $e\geq r_{e}$.

\textbf{Case 1.} Suppose $i <r_{e-s} $. Then $(m^i)^{-}=m^i$ and $ (k^{i}_{(m^i)^-})^-= k_{m^i}^{i}$. Therefore, 
\begin{align*}
    &\{ j \in J^{-}\setminus (\sqcup_{0 \leq r <i} (K^{r})^{-})\ |\ e(\Delta_j)=(m^{i})^- -1, b(\Delta_{j})< b(\Delta_{(k_{(m^{i})^-}^i)^-})\\
    =&\{ j \in J^{-}\setminus (\sqcup_{0 \leq r <i} K^{r})\ |\ e(\Delta_j)=m^{i} -1, b(\Delta_{j})< b(\Delta_{k_{m^{i}}^i})\}\\
    \subseteq &\{ j \in J\setminus (\sqcup_{0 \leq r <i} K^{r})\ |\ e(\Delta_j)=m^{i} -1, b(\Delta_{j})< b(\Delta_{k_{m^{i}}^i})\}.
\end{align*}
Since the last set is empty by Corollary \ref{cor char of Delta}(a)(4) for $K^i(\beta)$, this verifies Condition (4) in this case.

\textbf{Case 2.} Suppose $r_{l+1}> i \geq r_{l}$ for some $ e-s \leq l < e$. Then $(m^i)^{-}=m^{i+1}$ and $ (k^{i}_{(m^i)^-})^-= k_{m^{i+1}}^{i+1}$. Observe that $\sqcup_{0\leq r<i} (K^r)^{-} \subseteq \sqcup_{0\leq r <i+1} K^r$. Moreover, the difference set can be written down explicitly as follows
\[ ( \sqcup_{0\leq r <i+1} K^r) \setminus (\sqcup_{0\leq r<i} (K^r)^{-})= \{ k_{e-s}^{r_{e-s}}, \ldots, k_{l}^{r_{l}}  \} \sqcup \{ k_{l+1}^i,\ldots, k_{e}^i \}. \]
As a consequence, if $j \in J^{-} \setminus (\sqcup_{0\leq r<i} (K^r)^{-})$ but $j \not\in J \setminus ( \sqcup_{0\leq r <i+1} K^r) $, then $j \in \{ k_{l+1}^i,\ldots, k_{e}^i \}$. In particular, since $i+1 \leq r_{l+1}$, Lemma \ref{lem M-W index}(a) implies
\[e(\Delta_j) \geq l+1  \geq m^{r_{l+1}} \geq m^{i+1}> m^{i+1}-1.\]
Therefore,
\begin{align*}
    &\{ j \in J^{-}\setminus (\sqcup_{0 \leq r <i} (K^{r})^{-})\ |\ e(\Delta_j)=(m^{i})^- -1, b(\Delta_{j})< b(\Delta_{(k_{(m^{i})^-}^i)^-})\\
    =&\{ j \in J^{-}\setminus (\sqcup_{0 \leq r <i} (K^{r})^{-})\ |\ e(\Delta_j)=m^{i+1} -1, b(\Delta_{j})< b(\Delta_{k_{m^{i+1}}^{i+1}})\}\\
    = &\{ j \in J\setminus (\sqcup_{0 \leq r <i+1} K^{r})\ |\ e(\Delta_j)=m^{i+1} -1, b(\Delta_{j})< b(\Delta_{k_{m^{i+1}}^{i+1}})\},
\end{align*}
which is empty by Corollary \ref{cor char of Delta}(a)(4) for $K^{i+1}(\beta)$, this verifies Condition (4) in this case.

\textbf{Case 3.} Suppose $i\geq r_{e}$. Then $(m^i)^{-}=m^{i+1}$ and $ (k^{i}_{(m^i)^-})^-= k_{m^{i+1}}^{i+1}$. Similar to the previous case, we have 
\[ ( \sqcup_{0\leq r <i+1} K^r) \setminus (\sqcup_{0\leq r<i} (K^r)^{-})= \{ k_{e-s}^{r_{e-s}}, \ldots, k_{e}^{r_{e}}  \}, \]
and hence 
\[J^{-} \setminus (\sqcup_{0\leq r<i} (K^r)^{-})=J \setminus ( \sqcup_{0\leq r <i+1} K^r). \]
Thus similar argument in the previous case verifies Condition (4) in this case. This completes the verification of the claim and the proof of the lemma.
\end{proof}

Remark that we have identities
\[\delta_{d,a}= \delta_{ \half{a-d}, \half{a+d}, a},\ \delta_{a,d}= \widetilde{\delta}_{ \half{a-d}, \half{a+d}, a}.\]

As a corollary, this gives an alternate proof of the following fact.
\begin{lemma}
Consider local Arthur parameter of the form
\[ \psi= \bigoplus_{i\in I_{\rho}} \rho \otimes \Sym^{d_i} \otimes \Sym^{a_i},  \]
and let $\alpha= \delta_{\psi}$. Then $\widetilde{\alpha}= \delta_{\widehat{\psi}}$, where
\[ \widehat{\psi}= \bigoplus_{i\in I_{\rho}} \rho \otimes \Sym^{a_i} \otimes \Sym^{d_i}.  \]
\end{lemma}
\begin{proof}
We apply induction on $| I_{\rho}|$. Define $a+d, d$ as in Proposition \ref{prop main}. We apply Lemma \ref{lemma M-W} on $\alpha$, which contains a copy of $\delta_{d,a}= \delta_{ \half{a-d}, \half{a+d}, a}$. Note that the Assumptions are verified by Lemma \ref{lem a-d} with $\beta= \alpha$ and $i=0$. We obtain
\begin{align*}
    \widetilde{\alpha}&= \widetilde{\alpha^{-}}+ \widetilde{\delta}_{ \half{a-d}, \half{a+d}, a} \\
    &=\widetilde{\alpha^{-}}+ \delta_{a,d},
\end{align*}
where $\alpha^{-}=\alpha- \delta_{d,a}=\delta_{\psi^{-}}$ and
\[ \psi^{-}= \psi - \rho\otimes \Sym^{d}\otimes \Sym^{a}. \]

If $|I_{\rho}|=1$, then $\alpha^{-}$ is empty, and hence $\widetilde{\alpha}= \delta_{a,d}= \delta_{\widehat{\psi}}$. If $| I_{\rho}| >1$, then the induction hypothesis shows that $\widetilde{\alpha^{-}}=\delta_{\widehat{\psi^{-}}}$, which implies that $\widetilde{\alpha}= \delta_{\widehat{\psi}}$. This completes the proof of the lemma.
\end{proof}

Now we prove Proposition \ref{prop main}.
\begin{proof}
We use the notation $e=e(\beta), t=t(\beta), K^i=K^i(\beta)$ defined in Section \ref{sec lemma MW}.

By Lemma \ref{lem a-d}(b), $e= \half{a+d}$. Let 
\[r:= \min \{ 0 \leq i \leq t-1 \ | \ M(\beta^i)\supseteq [(-a+d)/2, (a+d)/2] \}.\]
Note that the right hand side is non-empty since 
\[\widetilde{\beta} \geq \widetilde{\alpha}= \delta_{\widehat{\psi}} \supseteq \delta_{a,d} \ni [(-a+d)/2, (a+d)/2].\] 
Recall our notation that $M(\beta^r)=[m^r, e]$ and $K^r=\{ k_{e}^r,\ldots, k^{r}_{m^r}\}$. Corollary \ref{cor char of Delta}(a)(3) implies 
\[ b(\Delta_{k_{m^r}^r})<b(\Delta_{k_{m^r+1}^r})< \cdots < b(\Delta_{k_{e}^r}), \]
and hence 
\begin{align}\label{eq 1}
    b(\Delta_{k_{m^r}^r}) \leq  b(\Delta_{k_e^r})- (e-m^r) \leq b(\Delta_{k_e^r})- \half{a+d}+ \half{-a+d}= b(\Delta_{k_e^r})-a,
\end{align}
where we use the condition $m^{r} \leq (-a+d)/2$ given by the containment $[m^r, e] \supseteq [(-a+d)/2,e]$. On the other hand, Lemma \ref{lem a-d}(b) gives
\begin{align}\label{eq 2}
    b(\Delta_{k_e^r}) \leq (a-d)/2.
\end{align} 
Therefore,
\[ b(\Delta_{k_{m^r}^r}) \leq \half{a-d} -a= \half{-a-d}. \]
By Lemma \ref{lem a-d}(a), the equality must hold, and hence all the inequalities in \eqref{eq 1} and \eqref{eq 2} are equalities. In particular, for $m^r \leq l \leq e$, we have $\Delta_{k_{l}^r}=[ l-d, l],$ We conclude that
\[ \beta= \{\Delta_{j}\}_{j\in J} \supseteq \{\Delta_{k^r}\}_{k^r\in K^r}=\{[l-d,l]\}_{\half{-a-d}\leq l \leq \half{a+d}} = \delta_{d,a}.\]
This proves Part (i).

For Part (ii), it is clear that $\alpha^{-}=\delta_{\psi^{-}}$ and $\beta^{-} \geq \alpha^{-}$. It remains to show that $\widetilde{\beta^{-}} \geq \widetilde{\alpha^{-}}$. We apply Lemma \ref{lemma M-W} on $\beta$, which contains a copy of $\delta_{d,a}= \delta_{\half{a-d},\half{a+d},a}$. Note that the Assumptions in Lemma \ref{lemma M-W} are verified by Lemma \ref{lem a-d} with $i=0$. We obtain
\[  \widetilde{\alpha}=\widetilde{\alpha^{-}}+\delta_{a,d},\   \widetilde{\beta}=\widetilde{\beta^{-}}+\delta_{a,d}, \]
which implies $\widetilde{\beta^{-}}\geq \widetilde{\alpha^{-}}$ since $ \widetilde{\beta} \geq \widetilde{\alpha}$. This completes the proof of Part (ii) and the proposition.
\end{proof}

\subsection{Proof of Lemma \ref{lem main intro}}\label{sec finish}

It is equivalent to prove Lemma \ref{lem main segment}. Write $\alpha=\delta_{\psi}$ where
\[ \psi= \bigoplus_{\rho \in \mathcal{C}_{\textrm{unit}}} \bigoplus_{i\in I_{\rho}} \rho \otimes \Sym^{d_i} \otimes \Sym^{a_i}.\]
Let 
\[ \psi_{\rho}:= \bigoplus_{i\in I_{\rho}} \rho \otimes \Sym^{d_i} \otimes \Sym^{a_i}.\]
We may decompose $\alpha= \sum_{\rho \in \mathcal{C}_{\textrm{unit}}} \alpha_{\rho}$, where $\alpha_{\rho}=\delta_{\psi_{\rho}}$, and $\beta$ decomposes similarly as $\beta = \sum_{\rho \in \mathcal{C}_{\textrm{unit}}} \beta_{\rho}$. The pair of inequalities $\beta \geq \alpha $ and $\widetilde{\beta} \geq \widetilde{\alpha} $ is equivalent to the pairs of inequalities $\beta_{\rho} \geq \alpha_{\rho}$, $\widetilde{\beta_{\rho}} \geq \widetilde{\alpha_{\rho}}$ for every $\rho$,. Therefore, we may assume 
\[ \psi=\psi_{\rho}= \bigoplus_{i\in I_{\rho}} \rho \otimes \Sym^{d_i} \otimes \Sym^{a_i}\]
for some $\rho \in \mathcal{C}_{\text{unit}}$, and adopt the notation in Proposition \ref{prop main}.

Apply induction on $k:=| I_{\rho}|$. When $k=1$, Proposition \ref{prop main}(i) implies $\alpha= \delta_{d,a}=\beta$. Suppose $k>1$, we construct $\alpha^{-},\beta^{-}$ as in Proposition \ref{prop main}(ii). Then the induction hypothesis implies $\alpha^{-}=\beta^{-}$, and hence
\[  \alpha= \alpha^{-}+ \delta_{d,a}= \beta^{-}+ \delta_{d,a}= \beta.\]
This completes the proof of the lemma.

\end{document}